%% file: forced_diameter.tex
\documentclass[12pt]{article}

\usepackage{amsmath,amsthm,amssymb}
\usepackage{setspace}
\usepackage{mathptmx}
\usepackage{subcaption}
\usepackage{bm}
\usepackage{tikz}
\pgfdeclarelayer{background}
\pgfdeclarelayer{foreground}
\pgfsetlayers{background,main,foreground}
\usepackage{url}

\theoremstyle{plain}
\newtheorem{theorem}{Theorem}

\newtheorem{corollary}[theorem]{Corollary}

\newtheorem{observation}[theorem]{Observation}

\newcommand{\lmaj}{\ensuremath{\succcurlyeq}}
\newcommand{\rmaj}{\ensuremath{\preccurlyeq}}

\newcommand{\fc}{\ensuremath{\mathcal{F}}}
\newcommand{\fb}{\ensuremath{\mathcal{B}}}

\newcommand{\seq}[1]{\ensuremath{\left( #1 \right)}}
\newcommand{\edge}[2]{\ensuremath{v_{#1}v_{#2}}}
\newcommand{\diam}{\operatorname{diam}}

\title{
Forced Edges and Graph Structure
\footnote{Official contribution of the National Institute of Standards and Technology; not subject to copyright in the United States.}
}

\author{Brian~Cloteaux \\
\small National Institute of Standards and Technology,\\
\small Applied and Computational Mathematics Division,\\
\small Gaithersburg, MD \\
\texttt{brian.cloteaux@nist.gov}}

\date {}

\begin{document}

\maketitle

\begin{abstract}
For a degree sequence, we define the set of edges that appear in
every labeled realization of that sequence as forced, while the edges
that appear in none as forbidden. We examine structure of graphs whose degree
sequences contain either forced or forbidden edges.
Among the things we show, we determine the structure of the forced or
forbidden edge sets, the relationship between the sizes of forced and forbidden
sets for a sequence, and the resulting structural consequences to their
realizations.  This includes showing that the diameter of every realization of
a degree sequence containing forced or forbidden edges is no greater
than 3, and that these graphs are maximally edge-connected.
\end{abstract}

\input{content.tex}

{ \singlespacing
\bibliographystyle{nist}
\bibliography{forced_diameter}
}

\end{document}

%% file: content.tex
\section{Introduction}
A degree sequence $\alpha$ is a sequence of  non-negative integers
$(\alpha_1, ..., \alpha_n)$  where 
there exists a simple and undirected graph $G$ whose node degrees correspond
with the values in $\alpha$.
For any simple, undirected graph $G = (V,E)$,
where $V$ is a set of vertices and $E$ is a set of edges, with node degrees
$\alpha$, $G$ is termed a {\it realization} of $\alpha$.
We use the standard notation of $n$ and $m$ to represent number of nodes
and edges respectively.  For this article, we assume that there is at
least one realization for each sequence $\alpha$, i.e. each sequence is
{\it graphic}.

A degree sequence may have a number of labeled realizations.  If an edge
appears between two labeled nodes for every realization, we denote 
that edge as {\it forced} for the degree sequence.
Conversely, if an edge never occurs in any labeled realization of $\alpha$,  we
denote that edge as {\it forbidden} for $\alpha$.
The simplest example of
a forced edge is when there is a dominating value in a sequence $\alpha$,
i.e. $\alpha_i = n-1$.  Then for every realization of $\alpha$, there
must be edges from the vertex $\alpha_i$ to every other vertex in the
graph.  Likewise, an empty value, where $\alpha_i = 0$, causes forbidden edges 
between $\alpha_i$ and  all the other vertices.  A non-trivial example of a
forced edge for a degree sequence is shown in
Figure \ref{fig:forced_edge_example}.

\begin{figure}
\begin{center}

\begin{subfigure}[b]{.4\textwidth}
\begin{tikzpicture}
\begin{scope}[every node/.style={circle,draw,fill=white,minimum size=1mm,inner sep=1pt}]
\node (v1) at (1,2) {$1$};
\node (v2) at (1,1) {$2$};
\node (v3) at (1,0) {$3$};
\node (v4) at (0,1) {$4$};
\node (v5) at (2,1) {$5$};
\node (v6) at (1,3) {$6$};
\end{scope}
\draw (v1) -- (v2);
\draw (v1) -- (v4);
\draw (v1) -- (v5);
\draw (v1) -- (v6);
\draw (v2) -- (v3);
\draw (v2) -- (v4);
\draw (v2) -- (v5);
\draw (v3) -- (v4);
\draw (v3) -- (v5);
\end{tikzpicture}
\end{subfigure}
\hspace{5mm}
\begin{subfigure}[b]{.4\textwidth}
\begin{tikzpicture}
\begin{scope}[every node/.style={circle,draw,fill=white,minimum size=1mm,inner sep=1pt}]
\node (v1) at (0,2) {$1$};
\node (v2) at (2,2) {$2$};
\node (v3) at (1,1) {$3$};
\node (v4) at (1,0) {$4$};
\node (v5) at (1,3) {$5$};
\node (v6) at (1,4) {$6$};
\end{scope}
\draw (v1) -- (v2);
\draw (v1) -- (v3);
\draw (v1) -- (v4);
\draw (v1) -- (v5);
\draw (v2) -- (v3);
\draw (v2) -- (v4);
\draw (v2) -- (v5);
\draw (v3) -- (v4);
\draw (v5) -- (v6);
\end{tikzpicture}
\end{subfigure}

\end{center}
\caption{ Every labeled realization of the degree sequence
$\alpha = \seq{4,4,3,3,3,1}$ is isomorphic to one of these two graphs. Since
there is always an edge between the two degree 4 nodes, then we say
that the edge $\edge{1}{2}$ as forced for $\alpha$.}
\label{fig:forced_edge_example}
\end{figure}
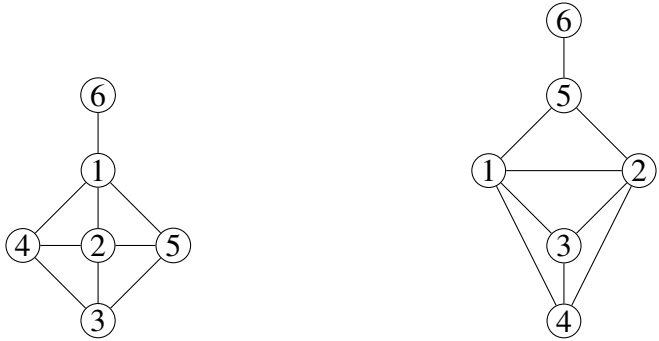

Consideration of the forced and forbidden edges for a degree sequence has both
algorithmic and theoretical applications.  For instance, the
creation of random graphs from a given degree sequence is useful for tasks
from counting graphs with a given degree sequence to creating models of
networks.  There are two principle approaches to creating a random
realization.  The most common method involves a Monte Carlo Markov Chain
(MCMC) approach, where one starts with a initial non-random realization
for a degree sequence, and then performing a random walk to a new
realization through a series of edge switches \cite{Greenhill:2015}.
Obviously, if a edge is forced, then it is in every realization, and can never
be swapped out during the random walk. For some realizations, it can
computationally expensive to find edges to switch.  We can optimize this edge
selection process by excluding forced edges.

A more striking example  is in the creation of a sequential
importance sampling (SIS) method for creating random realizations. In
the SIS approach, edge are randomly selected until a realization is
built.  The difficulty with this approach is that while selecting edges
it is possible to become stuck.  In other words, we can create a partial
graph in which it is impossible to complete into a realization for the
given degree sequence.  The first algorithm to overcome this difficulty
was proposed by Blitzstein and Diaconis \cite{Blitzstein:2010}.  Their
breakthrough idea was to show that by carefully selecting the edges
that are not forbidden, a realization can always be created.

A recent application area where forced edges provide a fundamental
limitation is in degree sequence packing \cite{Busch:2012,Yin:2016}.
The degree sequence packing problem is to determine whether for two
degree sequences, $\alpha$ and $\beta$, labeled realizations exist for
both sequences that are edge-disjoint. 
Obviously, if the two sequences contain the same forced edge then the
two sequences cannot pack.  For example, the sequence
$\seq{4,4,4,1,1,1,1,1,1}$
cannot be packed with sequence $\seq{0,1,1,0,0,0,0,0}$, since both sequences
have the forced edge $\edge{2}{3}$.

\section{Basic Definitions and Results}

We begin with some needed formal definitions and results.
A {\it degree sequence} $\alpha = \seq{\alpha_1, \alpha_2, ..., \alpha_{n}} $
is a set of non-negative integers such that
$n-1 \geq \alpha_1 \geq \alpha_2 \geq ... \geq \alpha_n \geq 0$. 
The complement of a sequence $\alpha$ is the sequence $\bar{\alpha}$
where $\bar{\alpha}_i = n - \alpha_{n+1-i} - 1$.
A sequence that corresponds to the vertex degrees of some simple graph is
called a {\it graphic} degree sequence.  A graph whose vertex degrees
match a degree sequence is termed a {\it realization} of that sequence.
To represent the degree sequence $\alpha$ of a given graph $G$,
we will use the notation $\deg(G)$ where $\deg(G) = \alpha$. 

For a realization $G=(V,E)$ of the sequence $\alpha$, 
we will use the notation $v_i$ to represent the vertex in $G$ whose
vertex degree corresponds to the $i$th value in the sequence $\alpha$, while
the neighborhood, or set of adjacent vertices, of $v_i$ is represented as
$N(v_i)$, i.e., $|N(v_i)| = \alpha_i$.
For a subset of vertices $S \subseteq V$,
the induced subgraph on this subset is represented as $G[S]$.
An edge between two vertices, $v_i$ and $v_j$, is designated as
$\edge{i}{j}$.
For a graph $G$, we denote the diameter of the graph as $\diam(G)$.

For a graphic degree sequence $\alpha$,
we define its {\it forbidden set} $\fb(\alpha)$ as the set
of all edges between labeled nodes that do not appear in any realization of
$\alpha$, while its {\it forced set} $\fc(\alpha)$ is the set of all edges
between labeled nodes that appear in every realization of $\alpha$.  
We will also define the set $\mathcal{P}(\alpha)$ to contain all the vertices
that are in some forced edge in $\alpha$, i.e.,
$\mathcal{P}(\alpha) = \lbrace v_i | \exists v_j : \edge{i}{j} \in \fc(\alpha) \rbrace$.

In order to compare degree sequences, we will use the following
partial ordering of {\it majorization}.
A degree sequence $\alpha$ majorizes (or dominates) the
integer sequence $\beta$,
denoted by $\alpha \succcurlyeq \beta$, if for all $k$ from $1$ to $n$
\begin{equation} \label{eqn:major1}
\sum_{i=1}^{k} \alpha_i \geq \sum_{i=1}^{k} \beta_i,
\end{equation}
and if the sums of the two sequences are equal.

A convenient fact that we will use is that the majorization order is
preserved by the complements of sequences, i.e., if $\alpha \lmaj \beta$
then $\bar{\alpha} \lmaj \bar{\beta}$.
\begin{theorem}
If $\alpha \lmaj \beta$, then $\bar{\alpha} \lmaj \bar{\beta}$.
\end{theorem}
\begin{proof}
For $k$ where $1 \leq k \leq n$, 
\begin{equation*}
\begin{split}
	\sum^{k}_{i=1} \bar{\alpha_i} &=
	\sum^{n}_{i=1} \bar{\alpha_i} - \sum^{n}_{i=k+1} \bar{\alpha_i} 
	\ =\  \sum^{n}_{i=1} \bar{\alpha_i} - \sum^{n}_{i=k+1} \left( (n-1) - \alpha_{n+1-i} \right) \\
	&= \sum^{n}_{i=1} \bar{\beta_i} - (n-k)(n-1) + \sum^{n-k}_{i=1}
	\alpha_{i} \\ 
	&\geq \sum^{n}_{i=1} \bar{\beta_i} - (n-k)(n-1) + \sum^{n-k}_{i=1}
	\beta_{i} \ =\ \sum^{k}_{i=1} \bar{\beta_i}. \\ 
\end{split}
\end{equation*}
\end{proof}

A degree sequence which has precisely one labeled realization is
called a {\it threshold sequence} and the resulting
realization is called a {\it threshold graph} \cite{Mahadev:1995}.
In the context of our discussion about forced edges, threshold
graphs can be seen as graphs where every edge is forced, and every
non-edge is forbidden.

For convenience, we introduce a notation for showing
increments or decrements to specific indices in a sequence. 
For the degree sequence $\alpha$, the sequences
$\ominus_{i_1,...,i_k} \alpha$ and $\oplus_{i_1,...,i_k} \alpha$ are defined by
\begin{equation}
\left( \ominus_{i_1,...,i_k} \alpha \right)_i = \left\{ \begin{array}{ll}
   \alpha_i - 1 & \textrm{for}\  i \in \{i_1,...,i_k\} \\
   \alpha_i & \textrm{otherwise,} \\
   \end{array} \right.
\end{equation}
\begin{equation}
\left( \oplus_{i_1,...,i_k} \alpha \right)_i = \left\{ \begin{array}{ll}
   \alpha_i + 1 & \textrm{for}\  i \in \{i_1,...,i_k\} \\
   \alpha_i & \textrm{otherwise.} \\
   \end{array} \right.
\end{equation}
There is a straightforward but nontrivial relationship between majorization
and the decrementing and incrementing operations. 

\begin{theorem}[Fulkerson and Ryser~\cite{Fulkerson:1962}, Lemma 3.1]
\label{thm:fulkerson-ryser}
If $\alpha \lmaj \beta$ and $\omega_1 = \lbrace i_1, ..., i_k \rbrace$
and $\omega_2 = \lbrace j_1, ..., j_k \rbrace$, where $i_1 \geq j_1,
..., i_k \geq j_k$ then
$\ominus_{\omega_1} \alpha \lmaj \ominus_{\omega_2} \beta$.
and
$\oplus_{\omega_2} \alpha \lmaj \oplus_{\omega_1} \beta$
\end{theorem}

For our purposes, the usefulness of comparing degree sequences using
majorization stems from the following result.

\begin{theorem}[Ruch and Gutman~\cite{Ruch:1979}, Theorem 1]
\label{thm:graphicality}
If the degree sequence $\alpha$ is graphic and $\alpha \lmaj \beta$,
then $\beta$ is graphic.
\end{theorem}

Finally, we will use another classic result.
\begin{theorem}[Kleitman and Wang~\cite{Kleitman:1973}, Theorem 2.1]
\label{thm:kleitman-wang}
For a degree sequence $\alpha$ and an index $i$, let $\alpha'$ be the
sequence created by subtracting 1 from the first $\alpha_i$ values in
$\alpha$ not including index $i$ and then setting $\alpha_i = 0$.
Then, the degree sequence $\alpha$ is graphic if and only if
the degree sequence $\alpha'$ is graphic.
\end{theorem}

\section{Forced and Forbidden Edges}

A simple observation about forced and forbidden edges set is that they
have a dual relationship through their complement degree sequences.

\begin{observation} \label{obs:dual}
\[ \fc(\alpha) = \fb(\bar{\alpha}). \]
\end{observation}

A method for determining whether an edge is either forced or forbidden for a
degree sequence is given by the next theorem.

\begin{theorem}[Blitzstein and Diaconis \cite{Blitzstein:2010}, Proposition 6.2]
\label{thm:forced_forbidden_edges} 
Let $\alpha$ be a graphic degree sequence and $i,j \in \{1,...,n\}$
with $i \neq j$. The edge $\edge{i}{j} \in \fc(\alpha)$ if and only if
$\oplus_{i,j} \alpha$ is not graphic, while the edge $\edge{i}{j} \in
\fb(\alpha)$ if and only if $\ominus_{i,j} \alpha$. 
\end{theorem}

\begin{proof}
The theorem can be viewed as a consequence of Kundu's Theorem~\cite{Kundu:1973}.
If $\oplus_{i,j} \alpha$ is graphic, then Kundu's Theorem guarantees that
there exists a realization $G$ of $\oplus_{i,j} \alpha$ containing the edge
$\edge{i}{j}$.  Removing the edge $\edge{i}{j}$ from $G$
shows that it is not forced in $\alpha$. The reverse direction is trivial.
The result for forbidden edges follows from the forced edge result on
the complement sequences.
\end{proof}

A useful fact for examining the structure of graphs containing forced
edges is that a force edge remains forced across
any induced subgraphs containing that edge.

\begin{theorem}
Let $\alpha$ be a graphic degree sequence and $i,j \in \{1,...,n\}$
with $i \neq j$. If the edge $\edge{i}{j} \in \fc(\alpha)$, then 
for any realization $G$ of $\alpha$ and for every vertex set $S$,
where
$\lbrace i,j \rbrace \subseteq S \subseteq V$,   $\edge{i}{j} \in
\fc(\deg(G[S]))$.
\end{theorem}

\begin{proof}
Assume that the edge $\edge{i}{j}$ is forced in
$G = (V_G,E_G)$
but not in the induced subgraph $G[S]$. Take a realization $R = (V_R, E_R)$
of the degree sequence of $G[S]$ that does not contain the edge
$\edge{i}{j}$, and create a new graph $ C= (V_G, E_C)$ where
for any two vertices $v_p,v_q \in V_G$, if $v_p,v_q \in S$ then
$\edge{p}{q} \in E_C$
if and only if $\edge{p}{q} \in E_R$;  else, $\edge{p}{q} \in E_C$
if and only if $\edge{p}{q} \in E_G$.  This graph
$C$ defines a realization of $\alpha$ that does not contain the
edge $\edge{i}{j}$ causing a contradiction. 
\end{proof}

There is a simple extension of Theorem \ref{thm:forced_forbidden_edges} for
sets of forbidden edges through the complement sequence.  This includes a test
for determining if an edge $\edge{i}{j}$ is forbidden by
testing whether or not  $\ominus_{i,j} \alpha$ is graphic.
We extend Theorem \ref{thm:forced_forbidden_edges} by using Theorem
\ref{thm:graphicality} to show an edge-inclusion result for forced and forbidden sets of
a degree sequence.
\begin{theorem}
\label{thm:sets_of_forced_edges}
For the graphic degree sequence $\alpha$, if $\edge{i}{j} \in
\fc(\alpha)$,
then for all indices $p,q$ where
$1 \leq p \leq i$ and $1 \leq q \leq j$ and $p \neq q$, $\edge{p}{q} \in \fc(\alpha)$.
\end{theorem}

\begin{proof}
Suppose that $\edge{p}{q}$ is not forced, then Theorem
\ref{thm:forced_forbidden_edges} implies that $\oplus_{p,q} \alpha$ is
graphic.  From Theorem \ref{thm:fulkerson-ryser}, it follows that
$\oplus_{p,q} \alpha \lmaj \oplus_{i,j} \alpha$,
and Theorem \ref{thm:graphicality} implies that
$\oplus_{i,j} \alpha$ is graphic, contradicting the assumption
that $\edge{i}{j}$ is forced.
\end{proof}

An immediate consequence of this proposition is that if there exist any forced
edges for a degree sequence, then the edge $\edge{1}{2}$ must be one of
them.
This gives linear-time methods to determine if a sequence has any forced
or forbidden edges by testing whether $\oplus_{1,2} \alpha$ or
$\ominus_{n-1,n} \alpha$ are graphic respectively. Extending this
observation
establishes conditions for degree sequences that cannot have forced
edges. 
\begin{theorem}
\label{thm:bounds}
For the graphic sequence $\alpha$ where $\alpha_n > 0$, if
\begin{equation}
n \geq \min \lbrace \frac{(\alpha_1 + \alpha_n + 2)^2}{4 \alpha_n},
\frac{(\alpha_1 + \alpha_n )^2}{2 \alpha_n} \rbrace,
\end{equation}
then $\fc(\alpha) = \emptyset$.
\end{theorem}

\begin{proof}
The first term in this bound comes by substituting $\alpha_1+1$ for the
maximum degree into the graphic bound
given by Zverovich and Zverovich (Theorem 6, \cite{Zverovich:1992});
it follows that if the above  bound holds then $\oplus_{1,2} \alpha$ is graphic.
The second term uses a degree sequence packing result. Using the
observation that two degree sequences cannot pack where both have the same
forced edge $\edge{1}{2}$, we apply Theorem 2.2 of Busch et
al. \cite{Busch:2012} to pack a sequence $\alpha$ with the
sequence $\seq{1,1,0,...,0}$ and after some algebraic manipulation establish
the second bound.
\end{proof}

Theorem \ref{thm:sets_of_forced_edges} is also enough to establish the
structure of the sets of forced and forbidden edges. While the induced
subsets $S \subseteq \mathcal{P}(\alpha)$ do not necessarily need to be
threshold graphs, the sets of forced edges for a degree sequence always
do form a threshold graph.

\begin{theorem}
\label{thm:threshold_structure}
For a graphic degree sequence $\alpha$, the graph $G = (\mathcal{P}(\alpha),
\fc(\alpha))$ is a threshold graph. 
\end{theorem}

\begin{proof}
We want to show that the induced subgraph on any four vertices in 
$G$ cannot be either $2K_2$, $P_4$, or $C_4$ thus showing that the set of
edges form a threshold graph \cite{Chvatal:1977}.
Select any two edges in $G$ having four unique vertices,  
$\edge{p}{q}$ and $\edge{r}{s}$.  Since the vertices are unique,
then we will assume without a loss of generality that $p < q$,  $r < s$, and
$p < r$.  From Theorem \ref{thm:sets_of_forced_edges}, the edge $\edge{p}{r}$
must also be forced so $2K_2$ cannot be induced.
If $p < q < r < s$ or $p < r < q < s$, then
Theorem \ref{thm:sets_of_forced_edges} guarantees that the edge $\edge{r}{q}$
is forced thus preventing $C_4$ and $P_4$ from being induced.
Similarly, if $p < r < s < q$, then the edge $\edge{p}{s}$ 
is forced for the same result, thus confirming the theorem.
\end{proof}

Over the set of partitions for some positive integer $p$,
majorization forms a lattice \cite{Brylawski:1973}.
In these partition lattices, at the top of the graphic sequences 
are the threshold sequences in which every edge is forced. In contrast, 
Theorem \ref{thm:sets_of_forced_edges} can be extended to show that the
regular sequences, which occupy the bottom of the lattice, 
cannot have any forced or forbidden edges
(other than trivially with the complete or empty sequences).
We formalize this observation by showing a strict 
ordering of forced and forbidden sets by subset down chains in this lattice.

\begin{theorem}[Barrus \cite{Barrus:2015}, Theorem 4.1]
\label{thm:forbidden_subsets}
For the graphic sequences $\alpha$ and $\beta$, if $\alpha \lmaj \beta$ then
$\fc(\alpha) \supseteq \fc(\beta)$ and $\fb(\alpha) \supseteq \fb(\beta)$.
\end{theorem}

\begin{proof}
From the assumption $\alpha \lmaj \beta$,
it follows from Theorem \ref{thm:fulkerson-ryser} that
$\oplus_{p,q} \alpha \lmaj \oplus_{p, q} \beta$.
If an edge $\edge{p}{q} \not\in \fc(\alpha)$ then $\oplus_{p,q} \alpha$
is graphic
and so $\oplus_{p,q} \beta$ must also be graphic.
Thus $\edge{p}{q} \not\in \fc(\beta)$ and so $\fc(\alpha) \supseteq
\fc(\beta)$.
The implication $\fb(\alpha) \supseteq \fb(\beta)$ immediately follows
from the complement sequences.
\end{proof}

\section{Structure of Realizations}

A useful result with structural implications is that forced
(and forbidden) edges for a degree
sequence imply independent sets (or cliques) in the realizations of the
degree sequence.

\begin{theorem}
\label{thm:forced_struct}
Let $\alpha$ be a degree sequence.
\begin{enumerate}
\item If $\edge{i}{j} \in \fc(\alpha)$, then for any realization $G$ of
	$\alpha$, the set of vertices $V - (N(i) \cup N(j))$ forms an
	independent set,
\item If $n-1 > \alpha_1$ and $\alpha_n > 0$ and
	$\edge{i}{j} \in \fb(\alpha)$, then for any realization $G$ of
	$\alpha$, the set of vertices $N(i) \cup N(j)$ forms a clique.
\end{enumerate}
\end{theorem}

\begin{figure}
\begin{center}

\begin{subfigure}[b]{.4\textwidth}
\begin{tikzpicture}
\begin{scope}[every node/.style={circle,draw,fill=white,minimum size=1mm,inner sep=1pt}]
\node (i) at (1,3) {$i$};
\node (j) at (4,3) {$j$};
\node (k) at (0,1) {$k$};
\node (l) at (2,1) {$l$};
\node (p) at (4,1) {$p$};
\end{scope}
\draw (i) -- (k);
\draw (i) -- (l);
\draw (l) -- (p);
\draw (j) -- (p);
\draw (k) to[out=-30,in=210] (p);
\end{tikzpicture}
\caption{Replace $\lbrace \edge{i}{l}, \edge{i}{k}, \edge{j}{p} \rbrace$ with
$\lbrace \edge{i}{j}, \edge{k}{l}, \edge{i}{p} \rbrace$. }
\label{fig:cliqueproof1}
\end{subfigure}
\hspace{5mm}
\begin{subfigure}[b]{.4\textwidth}
\begin{tikzpicture}
\begin{scope}[every node/.style={circle,draw,fill=white,minimum size=1mm,inner sep=1pt}]
\node (i) at (1,3) {$i$};
\node (j) at (4,3) {$j$};
\node (k) at (0,1) {$k$};
\node (l) at (2,1) {$l$};
\node (p) at (4,1) {$p$};
\node (q) at (0,3) {$q$};
\end{scope}
\draw (i) -- (k);
\draw (i) -- (l);
\draw (i) -- (p);
\draw (l) -- (p);
\draw (j) -- (p);
\draw (k) to[out=-30,in=210] (p);
\draw (q) -- (k);
\end{tikzpicture}
\caption{ Replace $\lbrace \edge{i}{l}, \edge{j}{p}, \edge{k}{q} \rbrace$ with
$\lbrace \edge{i}{j}, \edge{p}{q}, \edge{k}{l} \rbrace$ }
\label{fig:cliqueproof2}
\end{subfigure}

\begin{subfigure}[b]{.4\textwidth}
\begin{tikzpicture}
\begin{scope}[every node/.style={circle,draw,fill=white,minimum size=1mm,inner sep=1pt}]
\node (i) at (1,3) {$i$};
\node (j) at (4,3) {$j$};
\node (k) at (0,1) {$k$};
\node (l) at (2,1) {$l$};
\node (p) at (4,1) {$p$};
\node (q) at (5,1) {$q$};
\end{scope}
\draw (i) -- (k);
\draw (i) -- (l);
\draw (i) -- (p);
\draw (l) -- (p);
\draw (j) -- (p);
\draw (k) to[out=-30,in=210] (p);
\draw (q) -- (j);
\end{tikzpicture}
\caption{ Replace $\lbrace \edge{i}{p}, \edge{j}{q}\rbrace$ with
$\lbrace \edge{i}{j}, \edge{p}{q} \rbrace$ }
\label{fig:cliqueproof3}
\end{subfigure}
\hspace{5mm}
\begin{subfigure}[b]{.4\textwidth}
\begin{tikzpicture}
\begin{scope}[every node/.style={circle,draw,fill=white,minimum size=1mm,inner sep=1pt}]
\node (i) at (1,3) {$i$};
\node (j) at (4,3) {$j$};
\node (k) at (0,1) {$k$};
\node (l) at (2,1) {$l$};
\node (p) at (4,1) {$p$};
\node (q) at (5,1) {$q$};
\node (r) at (5,3) {$r$};
\end{scope}
\draw (i) -- (k);
\draw (i) -- (l);
\draw (i) -- (p);
\draw (l) -- (p);
\draw (j) -- (p);
\draw (k) to[out=-30,in=210] (p);
\draw (q) -- (r);
\end{tikzpicture}
\caption{ Replace $\lbrace \edge{i}{p}, \edge{j}{p}, \edge{q}{r}\rbrace$ with
$\lbrace \edge{i}{j}, \edge{p}{q}, \edge{p}{r} \rbrace$ }
\label{fig:cliqueproof4}
\end{subfigure}

\end{center}

\caption{This figure shows the cases used in the proof of Theorem
\ref{thm:forced_struct}.  For each of the above graphs, the edge
$\edge{i}{j}$ is not forbidden as shown by the  edge replacements in
each caption. }
\end{figure}
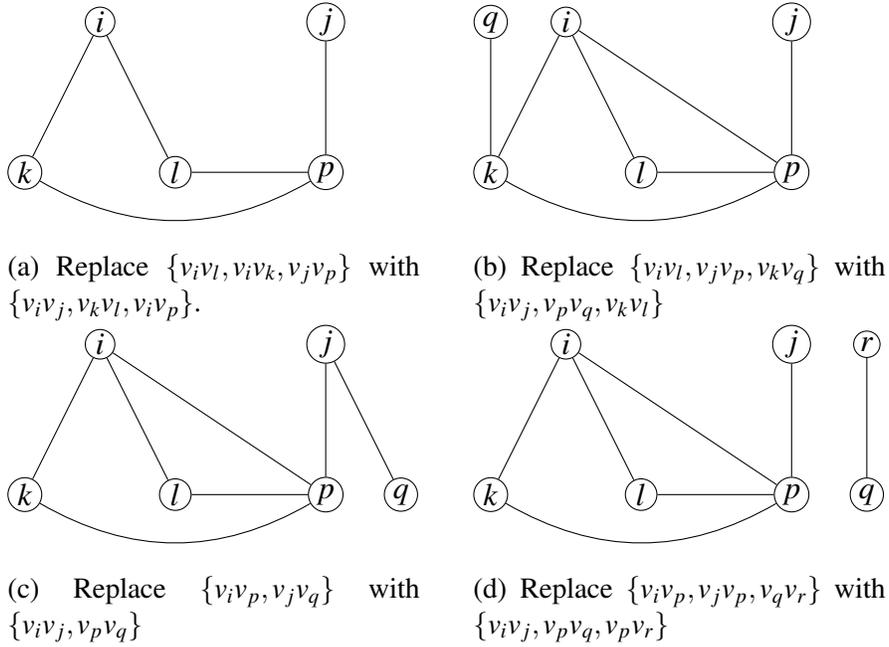

\begin{proof}
For the first statement, suppose there are vertices
$\{ v_p,v_q \} \subseteq V - (N(v_i) \cup N(v_j))$
that have an edge between them.  We can replace the
edges $\{ \edge{p}{q}, \edge{i}{j} \}$ with $\{ \edge{p}{i},
\edge{q}{j}\}$ forming a realization of $\alpha$ without the edge
$\edge{i}{j}$, causing a contraction.

For the second statement, suppose that $\edge{i}{j}$ is forbidden.
If $v_l \in N(v_i)$ and $v_p \in N(v_j)$ and
there is no edge $\edge{l}{p}$ a realization of $\alpha$, then $\lbrace
\edge{i}{l}, \edge{j}{p} \rbrace$ can be swapped out with $\edge{i}{j},
\edge{l}{p}$ contradicting the assumption that $\edge{i}{j}$ is forbidden.
Thus for $N(v_i) \cup N(v_j)$ not to be a clique requires that there exist two
vertices $v_l$ and $v_k$ where
$\{v_l,v_k \} \in N(v_i)$ but  $\{v_l,v_k \} \not\in N(v_j)$ and
$\edge{k}{l}$ is not in a realization.

Consider if the edge $\edge{i}{p}$ is also not in the realization, then we have
the induced graph shown in Figure \ref{fig:cliqueproof1}. As the caption
shows,  the edge $\edge{i}{j}$ would not be forbidden in this case.  Thus if
the vertex $v_p$ is connected to $\{i,j,k,l\}$ and since $\alpha_p < n-1$,
then there must exist a vertex $v_q$ in the realization that is not adjacent to
$v_p$.  

Again, since $\alpha_q > 0$, then $v_q$ must be adjacent to another vertex,
but we have already established that it cannot be adjacent to $v_i$.  Thus
either it is adjacent to $v_j$, $v_k$ (or $v_l$), or a completely separate
node $v_r$.  For each of those cases, Figures \ref{fig:cliqueproof2},
\ref{fig:cliqueproof3}, and \ref{fig:cliqueproof4} show that these arrangements
invalidates the assumption that $\edge{i}{j}$ is forbidden.
Thus the edge $\edge{k}{l}$ must exist, and the vertices $N(v_i) \cup N(v_j)$
form a clique. 
\end{proof}

Extending this results, we now relate the size of the sets of forced edges to
forbidden edges for a degree sequence, by showing that
forbidden edges in a degree sequence imply cliques of  forced
edges.

\begin{theorem}
\label{thm:forced_clique}
Let $\alpha$ be a graphic sequence where $\alpha_n > 0$.
If $\edge{i}{j} \in \fb(\alpha)$ then
there exists a clique of $\alpha_i$ nodes in $\fc(\alpha)$.
\end{theorem}

\begin{proof}
From Theorem~\ref{thm:forced_struct}, for any realization of
$\alpha$,  the vertices $N(i) \cup N(j)$ form a clique. 
Using Theorem \ref{thm:kleitman-wang}, we construct a realization $H$ of
$\alpha$  where the first $\alpha_i$ vertices are connected to $v_i$,
and so these first $\alpha_i$ vertices form a $K_{\alpha_i}$-clique. 
For the degree sequence $\eta$ created by removing the vertex $v_i$ and
its adjacent edges from $H$, this clique of the first $\alpha_i$ vertices
must exist in every realization, i.e., $K_{\alpha_i} \subseteq \fc(\eta)$.

Now take an arbitrary realization $G$ of $\alpha$. If we remove $v_i$ along with
its adjacent edges from $G$, then for the resulting graph $G'$ with its degree
sequence $\deg(G')$, it is straightforward to see that $\eta \rmaj \deg(G')$.
From Theorem \ref{thm:forbidden_subsets},
$\fc(\eta) \subseteq \fc(\deg(G'))$ and so any realization of
$\deg(G')$ must contain all the edges in $\fc(\eta)$, specifically
$K_{\alpha_i}$. By adding back
the vertex $v_i$, we see that every realization $G$ will also contain
those edges.
\end{proof}

We can extend this result to show forced cliques based on the minimum
degree value.

\begin{corollary} For the graphic sequence $\alpha$ where $\alpha_1 < n-2$ and 
$\fc(\alpha) \neq \emptyset$, then $\fc(\alpha)$ contains a clique of
size $\alpha_n$.
\label{thm:minforcedclique}
\end{corollary}

\begin{proof}
Applying Theorem \ref{thm:forced_clique} to the complement sequence
$\bar{\alpha}$, there must exist a forbidden set of size $n-1-\alpha_2$
in $\alpha$. Since $\alpha_1 < n-2$, then $|\fb(\alpha)| \geq 2$.  Thus
$\alpha_n$ must be in a forbidden edge with $\alpha_{n-1}$.  Then
applying Theorem \ref{thm:forced_clique} again, we arrive that
$\fc(\alpha)$ must contain a clique of size $\alpha_n$.
\end{proof}

We now show that having forced or forbidden edges for a degree sequence limits
the diameter of its realizations.

\begin{theorem} For the graphic sequence $\alpha$ where $\alpha_n \geq 1$, if
$\fc(\alpha) \neq \emptyset$, or $\fb(\alpha) \neq \emptyset$,
then for any realization $G$ of $\alpha$,
\begin{equation}
\diam (G) \leq 3.
\end{equation}
\label{thm:max_diam}
\end{theorem}

\begin{proof}
If $\alpha_1 = n-1$ then trivially $\deg(G) = 2$, thus we will assume that
$\alpha_1 < n-1$.
We begin with a consideration of the case when $\fc(\alpha) \neq \emptyset$.  
by partitioning the set of vertices of $G$ into three sets
where $V = \mathcal{P}(\alpha) \cup Q \cup R$.
We define the set $Q$ as the all the
vertices in $G$ that are adjacent to a vertex in $\mathcal{P}(\alpha)$, but are
not themselves in $\mathcal{P}(\alpha)$.  We next define the set $R$ as
all the remaining vertices, $R = V - \mathcal{P}(\alpha) - Q$.
By performing a case analysis, we show that for any two vertices $v_i$ and
$v_j$ in $G$, there is a path between them of length no greater than 3. 

\begin{description}
\item[Case $\lbrace v_i, v_j \rbrace \subseteq \mathcal{P}(\alpha)$: ]
Theorem \ref{thm:threshold_structure}
says that the forced edges between the vertices in $\mathcal{P}(\alpha)$ form
a connected threshold graph, implying
that the  minimum path length between any two vertices in
$\mathcal{P}(\alpha)$ is no more than 2.

\item[Case $v_i \in Q, v_j \in \mathcal{P}(\alpha)$: ]
From the definition of $Q$ and Theorem
\ref{thm:threshold_structure}, the path length between  $v_i$ and $v_j$
is no more than 3.

\item[Case $\lbrace v_i, v_j \rbrace \subseteq Q$: ]
Let $v_k \in N(v_i)$ and $v_l \in N(v_j)$ where $\lbrace v_k, v_l
\rbrace \subseteq \mathcal{P}(\alpha)$.
If $v_k = v_l$ or the edge $\edge{k}{l} \in E$, then we have
found a path of length no more than 3 between $v_i$ and $v_j$.
Else, from Theorem \ref{thm:threshold_structure} we can
find a path with length 2 composed of forced edges from $v_k$ to $v_l$;
let us assume that the path goes through $v_m$.
If the edge $\edge{i}{m} \not\in E$ then
we could replace the edges $\lbrace \edge{i}{k}, \edge{l}{m} \rbrace$
with $\lbrace \edge{i}{m}, \edge{k}{l}\rbrace$
violating the assumption that $\edge{l}{m} \in \fc(\alpha)$.  A similar argument
establishes that $\edge{j}{m}$ must also be in $G$, giving a
path of length 2 from $v_i$ to $v_j$ through $v_m$. 

\item[Case $v_i \in R, v_j \in \mathcal{P}(\alpha)$: ]
From Theorem \ref{thm:forced_struct}, since $N(v_i) \subseteq Q$, 
then any vertex $v_k \in N(v_i)$ that we choose will be in $Q$.
Now select two vertices $\lbrace v_m, v_n \rbrace \subseteq
\mathcal{P}(\alpha)$ such that $v_m \in N(v_k)$ and $\edge{m}{n} \in
\fc(\alpha)$.  We first note that $G$ also must contain the edge $\edge{k}{n}$,
because if $\edge{k}{n}$ did not exist
then we could replace the edges
$\lbrace \edge{i}{k}, \edge{m}{n} \rbrace$ in $G$ with the set  
$\lbrace \edge{i}{m}, \edge{k}{n} \rbrace$ violating the assumption that
$\edge{m}{n} \in \fc(\alpha)$.
Now because all the forced edges are connected, we can inductively extend this
argument to show that
every forced edge must be in a triangle with $v_k$.  Thus $v_i$ can
reach any vertex $v_j \in \mathcal{P}(\alpha)$ with a path of length 2,

\item[Case $v_i \in R, v_j \in Q$:]
The argument for proceeding case shows that $v_i$ can reach any vertex in
$Q$ with a
path of no more than length 3 by going through some vertex in
$\mathcal{P}(\alpha)$.

\item[Case $\lbrace v_i, v_j \rbrace \subseteq R$: ]
Choose two vertices $v_k \in N(v_i)$ and $v_l \in N(v_j)$ and an edge
$\edge{m}{n} \in \fc(\alpha)$. 
If $v_k = v_l$ then we found a path of length 2.
If not then the edge $\edge{k}{l}$ must be in $E$, or else 
we could replace the edges in $\lbrace \edge{i}{k},
\edge{j}{l}, \edge{m}{n}\rbrace$ with $\lbrace \edge{k}{l}, \edge{i}{m},
\edge{j}{n} \rbrace$ violating the assumption that $\edge{m}{n} \in
\fc(\alpha)$.
Thus there is a path of no more than length 3 between $v_i$ and $v_j$.

\end{description}

For the second part of the statement when $\alpha_n \geq 1$ and
$\fb(\alpha) \neq 0$, 
we note that Theorem \ref{thm:forced_clique} coupled with the proof of the first
part of Theorem \ref{thm:max_diam} is almost enough to prove the second
part; it only fails when except when the forbidden edges are strictly between
vertices of degree 1.
To show the complete statement, assume that $\edge{m}{n} \in \fb(\alpha)$.
Since $v_m$ and $v_n$ are not isolated, then we choose the vertices
$v_p \in N(v_m)$ and $v_q \in N(v_n)$ where $v_p$ and $v_q$ are not
necessarily distinct. 
Theorem \ref{thm:forced_struct} says that for every realization 
$G$ of $\alpha$, the vertices in $N(m) \cup N(n)$ form a clique.
If the diameter of the graph is greater than 3, then there would have to exist
a minimal 4-path in $G$ between two vertices $v_i$ and $v_j$. Without a loss of
generality, we can assume that neither the vertex $v_i$ nor its neighbor in
that path $v_k$ is in $N(m) \cup N(n)$, or else we could find a 3-path
from $v_i$ to $v_j$. 
But we could replace the edges in
$\lbrace \edge{i}{k}, \edge{m}{p}, \edge{n}{q} \rbrace$ with 
$\lbrace \edge{m}{n}, \edge{i}{p}, \edge{k}{q} \rbrace$ violating the
assumption that $\edge{m}{n} \in \fb(\alpha)$.  
\end{proof}

We now examine the edge connectivity of a graph whose degree sequence
contains either a forced or forbidden edge. 
The edge connectivity $\lambda(G)$ is 
the minimum cardinality of an edge-cut over all edge-cuts of $G$.
There is a trivial upper bound for $\lambda(G) \leq \alpha_n$ where
$\alpha = \deg(G)$.  When a graph $G$ has this edge connectivity of
$\lambda(G) = \alpha_n$, then it is said to be maximally edge-connected. 
Any realization of a degree sequence with either forced or forbidden edges is
maximally edge-connected.

\begin{theorem} For the graphic sequence $\alpha$ where $\alpha_n \geq 1$,
if $\fb(\alpha) \neq \emptyset$ or $\fc(\alpha) \neq \emptyset$,
then for any realization $G$ of $\alpha$,
\begin{equation}
\lambda(G) = \alpha_n.
\end{equation}
\end{theorem}

\begin{proof}
We begin with some simple observations about what is required for a
graph to be maximally edge-connected.  If $\alpha_n = 1$, then for the
connected graph $G$, $\lambda(G) = \alpha_n$ is trivially true;
thus we assume that $\alpha_n \geq 2$.  In addition,
a result by Plesn\'ik~\cite{Plesnik:1975} establishes that if $\diam(G)
\leq 2$, then is maximally edge-connected.  Thus, from Theorem
\ref{thm:max_diam}, if $G$ is not maximally edge-connected,
then $\diam(G) = 3$.

For a contradiction, we assume that there is a realization of $G$ where
$\lambda(G) < \alpha_n$.
We denote the edge set $S$ as an arbitrary minimum edge-cut of $G$, and
the two components of $G$ with $S$ removed as $P$ and $Q$.  For each set
$P$ and $Q$, we partition each into two sets, $P = P_s \cup P_n$ (or
$Q = Q_s \cup Q_n$), where $P_s$ (or $Q_s$) is the set of vertices  in
$P$ (or $Q$)
with an adjacent edge in $S$, and $P_n$ (or $Q_n$) are the remaining vertices.

Using an argument first presented by Hellwig and Volkmann \cite{Hellwig:2008},
we show that $|P_n| \geq 2$. From the assumption that
$\lambda(G) \leq \alpha_n - 1$, then
\begin{equation}
\alpha_n |P| \leq \sum_{p \in P} \deg(p) \leq |P|(|P|-1) + \alpha_n -1,
\end{equation}
which implies that $|P| \geq \alpha_n + 1$.  Along with the assumption
that $|P_s| \leq \lambda(G) \leq \alpha_n -1$, it follows that
$|P_n| = |P| - |P_s| \geq 2$. There is a similar argument to show that
$|Q_n| \geq 2$ also.  One implication from this  result is that since $G$ is
not maximally edge-connected, then $\alpha_1 \leq n-3$.

We now show that if $\edge{i}{j} \in \fc(\alpha)$,
then $v_i$ and $v_j$ must be in separate components.
Suppose that $\lbrace v_i,v_j \rbrace \subseteq Q$, then
from the proceeding argument, there must be at least one edge
$\edge{k}{l}$ strictly in $P$.  This edge $\edge{k}{l}$
would allow us to replace $\lbrace \edge{i}{j}, \edge{k}{l} \rbrace$ with
$\lbrace \edge{i}{k}, \edge{j}{l} \rbrace$
violating the assumption that $\edge{i}{j}$ is forced; thus, each forced
edge must be in $S$.
Extending this observation shows that if $K_3 \subseteq \fc(\alpha)$, we would
have a contradiction with $G$ not being maximally edge-connected.

\begin{figure}
\begin{center}
\begin{subfigure}[b]{.3\textwidth}
\begin{tikzpicture}
\newcommand{\background}[5]{%
  \begin{pgfonlayer}{background}
    \path (#1.west |- #2.north)+(-0.2,0.2) node (a1) {};
    \path (#3.east |- #4.south)+(+0.2,-0.1) node (a2) {};
    \path[fill=black!30,rounded corners, draw=black!50, dashed]
      (a1) rectangle (a2);
    \path (a1.east |- a1.south)+(0.8,-0.2) node (u1) {\scriptsize\textit{$#5$}};
  \end{pgfonlayer}}
\begin{scope}[every node/.style={circle,draw,fill=white,minimum size=1mm,inner sep=1pt}]
\node (i) at (1,1) {$i$};
\node (j) at (2,1) {$j$};
\node (k) at (0,2) {$k$};
\node (l) at (0,0) {$l$};
\node (p) at (3,1) {$p$};
\end{scope}
\draw (i) -- (j);
\draw (j) -- (p);
\draw (k) -- (l);
\draw (k) -- (i);
\draw (l) -- (i);
\background{k}{k}{i}{l}{P}
\background{j}{k}{p}{l}{Q}
\end{tikzpicture}
\caption{Replace  $\lbrace \edge{k}{l}, \edge{i}{j}, \edge{j}{p} \rbrace$
	with $\lbrace \edge{k}{j}, \edge{j}{l}, \edge{i}{p} \rbrace$ }
\label{fig:case1}
\end{subfigure}
\hspace{4mm}
\begin{subfigure}[b]{.3\textwidth}
\begin{tikzpicture}
\newcommand{\background}[5]{%
  \begin{pgfonlayer}{background}
    \path (#1.west |- #2.north)+(-0.2,0.2) node (a1) {};
    \path (#3.east |- #4.south)+(+0.2,-0.1) node (a2) {};
    \path[fill=black!30,rounded corners, draw=black!50, dashed]
      (a1) rectangle (a2);
    \path (a1.east |- a1.south)+(0.8,-0.2) node (u1) {\scriptsize\textit{$#5$}};
  \end{pgfonlayer}}
\begin{scope}[every node/.style={circle,draw,fill=white,minimum size=1mm,inner sep=1pt}]
\node (i) at (1,1) {$i$};
\node (j) at (2,1) {$j$};
\node (k) at (0,2) {$k$};
\node (l) at (0,0) {$l$};
\node (p) at (3,1) {$p$};
\end{scope}
\draw (i) -- (j);
\draw (j) -- (p);
\draw (k) -- (l);
\draw (k) -- (i);
\background{k}{k}{i}{l}{P}
\background{j}{k}{p}{l}{Q}
\end{tikzpicture}
\caption{Replace $\lbrace \edge{k}{l}, \edge{i}{j}
\rbrace$ with $\lbrace \edge{k}{j}, \edge{i}{l} \rbrace$ }
\label{fig:case2}
\end{subfigure}
\hspace{5mm}
\begin{subfigure}[b]{.3\textwidth}
\begin{tikzpicture}
\newcommand{\background}[5]{%
  \begin{pgfonlayer}{background}
    \path (#1.west |- #2.north)+(-0.2,0.2) node (a1) {};
    \path (#3.east |- #4.south)+(+0.2,-0.1) node (a2) {};
    \path[fill=black!30,rounded corners, draw=black!50, dashed]
      (a1) rectangle (a2);
    \path (a1.east |- a1.south)+(0.8,-0.2) node (u1) {\scriptsize\textit{$#5$}};
  \end{pgfonlayer}}
\begin{scope}[every node/.style={circle,draw,fill=white,minimum size=1mm,inner sep=1pt}]
\node (i) at (1,1) {$i$};
\node (j) at (2,1) {$j$};
\node (k) at (0,2) {$k$};
\node (l) at (0,0) {$l$};
\node (p) at (3,1) {$p$};
\end{scope}
\draw (i) -- (j);
\draw (j) -- (p);
\draw (k) -- (l);
\background{k}{k}{i}{l}{P}
\background{j}{k}{p}{l}{Q}
\end{tikzpicture}
\caption{Replace $\lbrace \edge{k}{l}, \edge{i}{j} \rbrace$ with
$\lbrace \edge{k}{j}, \edge{i}{l} \rbrace$ }
\label{fig:case3}
\end{subfigure}
\end{center}
\caption{The three possible cases when $\edge{k}{l} \in E$ and $\lbrace
	k,l \rbrace \subseteq P_n$.  In all three cases, the edge
$\edge{i}{j}$ is not forced causing a contradiction.}
\label{fig:edge_pn_cases}
\end{figure}
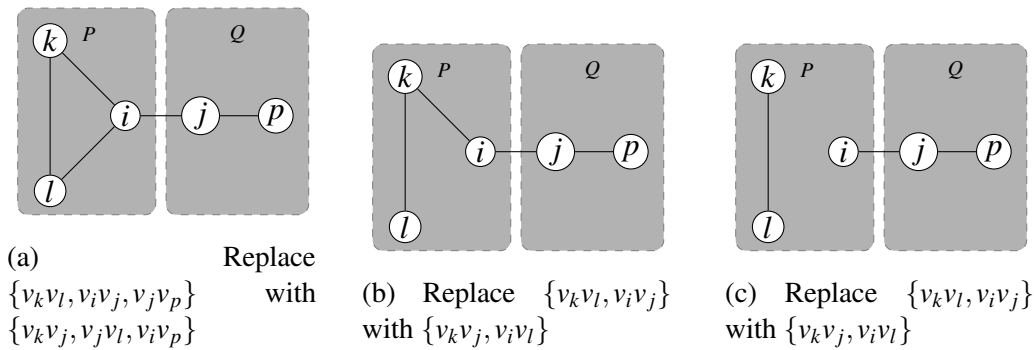

Let us consider the case where the forbidden edge set for $\alpha$ is not
empty, $\fb(\alpha) \neq \emptyset$.
From Theorem \ref{thm:forced_clique}, since $\alpha_n \geq 2$,
then $G$ has a clique of $\alpha_n$ in $\fc(\alpha)$, and so
if $\alpha_n > 3$, then $K_3 \in \fc(\alpha)$ proving that $G$ is not
maximally edge-connected.  Thus the only possible case for $\alpha$ not
covered by this result is when
$\alpha_n = 2$ and the resulting forced edge $\edge{i}{j}$ makes up the set $S$.
Assuming that $\lbrace v_k, v_l \rbrace \subseteq P_n$ and $v_p \in Q_n$, then
if there would exist an edge between $v_k$ and $v_l$ the induced subgraph
$G[\lbrace v_i, v_j, v_k, v_l, v_p \rbrace]$ would be one of the three cases in
Figure \ref{fig:edge_pn_cases}. Since in all three cases the edge
$\edge{i}{j}$ is not forced, then the edge $\edge{k}{l}$ cannot
exist.  This means that in general that any vertex in $P_n$
(or $Q_n$) must be connected to members of $P_s$ (or $Q_s$) only,
and specifically, in this case, $\deg(v_k) = \deg(v_l)= 1$.
This is a contradiction to $\alpha_n \geq 2$, and thus
it  follows that if $\fb(\alpha) \neq \emptyset$, then $\lambda(G) = \alpha_n$.

When the forced edge set is not empty, we again use Theorem
\ref{thm:forced_clique}, this time on the complement sequence
$\bar{\alpha}$, to show that
forbidden edge set $\fb(\alpha)$ has clique of size $n-1-\alpha_1$.
Since $\alpha_1 \leq n-3$, then the forbidden edge set is not empty,
and thus to avoid a contradiction, then $G$ must be maximally edge-connected.
\end{proof}